\documentclass[11pt]{amsart} \usepackage{latexsym, amssymb, stmaryrd, amsthm, amscd, amsmath}
\usepackage[T1]{fontenc}

\newtheorem{thm}{Theorem}
\newtheorem{lemma}[thm]{Lemma}
\newtheorem{prop}[thm]{Proposition}
\newtheorem{cor}[thm]{Corollary}

\newtheorem{fact}[thm]{Fact}

\theoremstyle{definition}

\newtheorem{rmk}[thm]{Remark}

\theoremstyle{remark}

\makeatletter

\def\dotminussym#1#2{%
  \setbox0=\hbox{$\m@th#1-$}%
  \kern.5\wd0%
  \hbox to 0pt{\hss\hbox{$\m@th#1-$}\hss}%
  \raise.6\ht0\hbox to 0pt{\hss$\m@th#1.$\hss}%
  \kern.5\wd0}
\newcommand{\dotminus}{\mathbin{\mathpalette\dotminussym{}}}
\renewcommand{\r}{\mathbb{R}}
\newcommand{\Z}{\mathbb{Z}}

\newcommand{\curly}[1]{\mathcal{#1}}

\newcommand{\B}{\curly{B}}

\newcommand{\la}{\curly{L}}

\renewcommand{\to}{\rightarrow}

\newcommand{\cstar}{$\mathrm{C}^*$}

\newcommand{\e}{\epsilon}

\def \K {{\mathcal K}}

\def \<{\langle}
\def \>{\rangle}

\def \*Z {{{^*}\Z}}
\def \((  {(\!(}
\def \)) {)\!)}

\numberwithin{equation}{section}

\def \u{\mathcal U}

\def \os{\operatorname{os}}

\newcommand{\tom}[1]{\textbf{\Large !!}{\footnotesize[[TS- #1]]}}

\allowdisplaybreaks[2]

\DeclareMathOperator{\EX}{\mathcal{EX}}

\title{On the axiomatizability of \cstar-algebras as operator systems}
\author{Isaac Goldbring and Thomas Sinclair}
\thanks{Goldbring's work was partially supported by NSF CAREER grant DMS-1349399.}

\address {Department of Mathematics, Statistics, and Computer Science, University of Illinois at Chicago, Science and Engineering Offices M/C 249, 851 S. Morgan St., Chicago, IL, 60607-7045}
\email{isaac@math.uic.edu}
\urladdr{http://www.math.uic.edu/~isaac}

\address{Mathematics Department, Purdue University, 150 N. University Street, West Lafayette, IN 47907-2067}
\email{tsincla@purdue.edu}
\urladdr{http://www.math.purdue.edu/~tsincla/}

\begin{document}

\begin{abstract}
We show that the class of unital \cstar-algebras is an elementary class in the language of operator systems.  As a result, we have that there is a definable predicate in the language of operator systems that defines the multiplication in any \cstar-algebra.  Moreover, we prove that the aforementioned class is $\forall\exists\forall$-axiomatizable but not $\forall\exists$-axiomatizable nor $\exists\forall$-axiomatizable.
\end{abstract}

\maketitle


Recall that a \cstar-algebra is a $^*$-subalgebra of $\B(H)$, the $^*$-algebra of bounded operators on a complex Hilbert space, that is closed in the operator norm topology.  In this note, we assume that all \cstar-algebras are unital, namely that they contain the identity operator.  As shown in \cite[Proposition 3.3]{modeloperator2}, there is a natural (continuous) first-order language $\la_{\mathrm{C}^*}$ in which $\K_{\mathrm{C}^*}$, the class of $\la_{\mathrm{C}^*}$-structures that are unital \cstar-algebras, is an elementary class, meaning that there is a (universal) $\la_{\mathrm{C}^*}$-theory $T_{\mathrm{C}^*}$ for which $\K_{\mathrm{C}^*}$ is the class of models of $T_{\mathrm{C}^*}$; in symbols, $\K_{\mathrm{C}^*}=\operatorname{Mod}(T_{\mathrm{C}^*})$.  (The authors only treat not necessarily unital \cstar-algebras, but one just adds a constant to name the identity with no additional complications.)

An operator system is a $^*$-closed subspace of $\B(H)$ that is closed in the operator norm topology.  The appropriate morphisms between operator systems are the unital completely positive linear maps.  There is a natural first-order language $\la_{\os}$ in which the class of operator systems is universally axiomatizable; see \cite[Subsection 3.3]{6authors} and \cite[Appendix B]{KEP}.  Since the operator system structure on a C$^*$-algebra is uniformly quantifier-free definable, we may assume that $\la_{\os}\subseteq \la_{\mathrm{C}^*}$.  For a \cstar-algebra $A$, we let $A|\la_{\os}$ denote the reduct of $A$ to $\la_{\os}$, which simply means that we view $A$ merely as an operator system rather than as a \cstar-algebra.  Set $\K_{\mathrm{C}^*}|\la_{\os}:=\{A|\la_{\os} \ : \ A\in \K_{\mathrm{C}^*}\}$.  In \cite{GL}, the following question was raised:  is $\K_{\mathrm{C}^*}|\la_{\os}$ an elementary class?  The main result of this note is to give an affirmative answer to this question.


We first need a lemma, which is nearly identical to \cite[Theorem 6.1]{BN}.  Some notation:  for a \cstar-algebra $B$ and $x,y,z,b\in B$, let $\varphi(x,y,z,b)$ be the $\la_{\os}$-formula 
$$\left | \left \| \begin{matrix} 0 & y & 1 & 0 \\ 2\cdot 1 & x & z & b \end{matrix}\right\|^2-\left \|\begin{matrix} 2\cdot 1 & x & z & b\end{matrix}\right\|^2\right |$$

\begin{lemma}\label{mainlemma}
Suppose that $A$ is a subsystem of the unital \cstar-algebra $B$.  Then $A$ is closed under products if and only if we have
$$\sup_{x,y\in A_1}\inf_{z\in A_1}\sup_{b\in B_{2}}\varphi(x,y,z,b)=0.$$
\end{lemma}

\begin{proof} Fix $\e\in (0,1)$ and $x,y\in A_1$. Choose $z\in A_1$ such that $$\sup_{b\in B_{2}} \varphi(x,y,z,b) < \e$$ and let $b$ be the square root of $\|xx^* + zz^*\|\cdot 1 - xx^* - zz^*\in B$ so that $\|b\|^2 = \|b^2\|\leq \|xx^* + zz^*\|\leq 2$. Multiplying $[\begin{matrix} 2\cdot 1 & x & z & b\end{matrix}]$ by its adjoint have that $$\left \|\begin{matrix} 2\cdot 1 & x & z & b\end{matrix}\right\|^2 = 4\cdot 1 + xx^* + zz^* + bb^* = (4 + \|xx^* + zz^*\|)\cdot 1.$$ Similarly, it holds that $$\left \| \begin{matrix} 0 & y & 1 & 0 \\ 2\cdot 1 & x & z & b \end{matrix}\right\|^2 = \left\| \begin{matrix} 1 + yy^* & yx^* + z^*\\ xy^* + z & (4 + \|xx^* + zz^*\|)\cdot 1\end{matrix} \right\|.$$ Examining the second row, it follows that the norm of the right side is at least $(\|xy^* + z\|^2 + (4 + \|xx^* + zz^*\|)^2)^{1/2}$, whence $\|xy^* + z\|\leq 4\sqrt\e$. As $A$ is complete, it follows that $xy^*\in A$.

Conversely, the above calculations show that if $A$ is multiplicatively closed, then setting $z = -xy^*$ suffices.
\end{proof}

\begin{thm} Let $B$ be a \cstar-algebra. If $A\subset B$ is a subsystem of $B$ which is an elementary substructure in the language of operator systems, then $A$ is a \cstar-subalgebra of $B$.
\end{thm}

\begin{proof} Since $A$ inherits the unit of $B$ and $A$ is self-adjoint, we need only check that $A$ is closed under products, that is, we need only verify the condition of the previous lemma.  Fixing $x,y\in A_1$,  we have
$$\inf_{z\in B_1}\sup_{b\in B_{2}}\varphi(x,y,z,b)=0.$$  Since $A$ is an elementary substructure of $B$, we have
$$\inf_{z\in A_1}\sup_{b\in A_{2}}\varphi(x,y,z,b)=0.$$  Fix $\epsilon>0$ and take $z\in A_1$ such that
$$\sup_{b\in A_{2}}\varphi(x,y,z,b)\leq \epsilon.$$  By elementarity again, we have
$$\sup_{b\in B_{2}}\varphi(x,y,z,b)\leq \epsilon,$$ whence we have
$$\inf_{z\in A_1}\sup_{b\in B_{2}}\varphi(x,y,z,b)=0,$$ which is what we desired.  
\end{proof}

\begin{cor}\label{maintheorem}
$\K_{\mathrm{C}^*}|\la_{\os}$ is an elementary class.
\end{cor}
\begin{proof}
We use the semantic test for axiomatizability, that is, we show that $\K_{\mathrm{C}^*}|\la_{\os}$ is closed under isomorphism, ultraproduct, and ultraroot.  (See \cite[Proposition 5.14]{bbhu}.)  Closure under isomorphism and ultraproducts is clear.  We thus only need to show that $\K_{\mathrm{C}^*}|\la_{\os}$ is closed under ultraroots.  So suppose that $A$ is an $\la_{\os}$-structure for which $A^\u\in \K_{\mathrm{C}^*}|\la_{\os}$.  Note that this implies that $A$ is an operator subsystem of $A^\u$, whence $A$ is an elementary substructure by {\L}os' theorem. Thus $A$ is a \cstar-subalgebra of $A^\u$.
\end{proof}


We recall the \emph{Beth Definability Theorem} (see, for example, \emph{the proof of }\cite[Theorem 3.5.1]{nuclear}):

\begin{fact}
Suppose that $T$ is a theory in a language $\la$ and $T\subseteq T'$ where $T'$ is a theory in a language $\la'$ with no new sorts.  Further suppose that, for any model $M$ of $T'$ and any automorphism $\sigma$ of $M|\la$, $\sigma$ is also an automorphism of $M$.  Then every predicate in $\la'$ is $T'$-equivalent to a definable predicate in $\la$, that is, there is a sequence $\varphi_n(\vec x)$ of $\la$-formulae such that, for all $A\in \operatorname{Mod}(T')$ and all $\vec a\in A$, we have that $\varphi_n^A(\vec a)$ converges uniformly to $P^A(\vec a)$. 
\end{fact}

Let $T_{\mathrm{C}^*,\os}$ be an $\la_{\os}$-theory such that $\K_{\mathrm{C}^*}|\la_{\os}=\operatorname{Mod}(T_{\mathrm{C}^*,\os})$.  A well-known consequence of Pisier's Linearization Trick (see Corollary \ref{forgetful} below for a short proof) implies that any complete order isomorphism between two \cstar-algebras is actually a ${}^*$-isomorphism.  We thus have:

\begin{cor}
There is an $\la_{\os}$-definable predicate $P$ such that, for any \cstar-algebra $A$ and any $x,y,z\in A_1$, we have
$$P^A(x,y,z)=d(x\cdot y,z).$$
\end{cor}

We now consider the quantifier-complexity of an axiomatization for $\K_{\mathrm{C}^*}|\la_{\os}$.  Recall that if $X$ is an operator system and $u\in X$, then $u$ is called a unitary of $X$ if $u$ is a unitary of the \cstar-envelope $\mathrm{C}^*_e(X)$.

\begin{prop}
$\K_{\mathrm{C}^*}|\la_{\os}$ is not $\forall\exists$-axiomatizable.
\end{prop}

\begin{proof}
If $\K_{\mathrm{C}^*}|\la_{\os}$ were $\forall\exists$-axiomatizable, then there would be $A\in \K_{\mathrm{C}^*}|\la_{\os}$ that is existentially closed for $\K_{\mathrm{C}^*}|\la_{\os}$.  However, in \cite[Section 5]{GL}, it was observed that if $\phi:X\to Y$ is a complete order embedding that is also existential, then $\phi$ maps unitaries to unitaries.  Take a complete order embedding of $A$ into a \cstar-algebra $B$ that is not a $^*$-homomorphism (see, for example, \cite[Section 5]{GL}); since this embedding maps unitaries to unitaries (since $A$ is existentially closed for $\K$), this contradicts Corollary \ref{pisier}.
\end{proof}

\begin{rmk}
In \cite[Section 5]{GL}, it was observed that if the set of unitaries in operator systems is definable (uniformly over all operator systems), then $\K_{\mathrm{C}^*}|\la_{\os}$ is elementary.  (This is equivalent to the statement that if $(X_i)$ is a family of operator systems and $u$ is a unitary of $\prod_\u X_i$, then there are unitaries $u_i$ in $X_i$ such that $u=(u_i)^\bullet$.)  In fact, if unitaries in operator systems were definable, then the following axioms would axiomiatize $\K_{\mathrm{C}^*}|\la_{\os}$:
\begin{itemize}
\item[(i)] the universal axioms for operator systems;
\item[(ii)] $\sup_{x\in A_1}\inf_{u_1,\ldots,u_4\in U(A)} d(x,\frac{1}{2}(u_1+\cdots+u_4))=0.$
\item[(iii)] $$\sup_{u,v\in U(A)}\inf_{x\in A}d\left(\left(\begin{matrix} 1 & u & x \\ u^* & 1 & v \\ x^* & u^* & 1\end{matrix}\right),M_3(A)_+\right)=0.$$
\end{itemize}
To see this, first note that an operator system is a \cstar-algebra if and only if it is spanned by its unitaries and is closed under the product of two unitaries.  The axiom in (ii) implies that the unitaries span; moreover, in a \cstar-algebra, any self-adjoint element is an average of two unitaries, whence the axiom in (ii) is true in any \cstar-algebra.  Moreover, a result of Walter (\cite{walter}) shows that the matrix appearing in axiom (iii) is positive if and only if $x=uv$.

While we still do not know if the unitaries in operator systems are definable, the previous proposition shows that \emph{if} they are definable, then they are not definable by an $\inf$-predicate, for otherwise the above axioms would give an $\forall\exists$-axiomatization of $\K_{\mathrm{C}^*}|\la_{\os}$.  This latter statement would have the following consequence:  there is an inclusion $X\subseteq Y$ of operator systems and an element $x\in X$ such that the distance from $x$ to the unitaries in $X$ is larger than the distance from $x$ to the unitaries in $Y$.
\end{rmk}

\begin{prop}
$\K_{\mathrm{C}^*}|\la_{\os}$ is not $\exists\forall$-axiomatizable.
\end{prop}

\begin{proof}
Fix a \cstar-algebra $A$ and an operator system $X$ that is \textbf{not} a \cstar-algebra with $A\subseteq X\subseteq A^\u$.  Suppose, towards a contradiction, that $\K_{\mathrm{C}^*}|\la_{\os}$ is $\exists\forall$-axiomatizable and let $\sigma:=\inf_x\sup_y\varphi(x,y)$ be such an axiom.  Fix $\epsilon>0$ and take $a\in A$ such that $(\sup_y\varphi(a,y))^A\leq \epsilon$.  It follows that $(\sup_y \varphi(a,y))^{A^\u}\leq \epsilon$, whence $(\sup_y \varphi(a,y))^X\leq \epsilon$.  Since $a\in X$ and $\epsilon>0$ was arbitrary, we see that $\sigma^X=0$.  Since $\sigma$ was an arbitrary axiom, we see that $X$ is a \cstar-algebra, yielding a contradiction.
\end{proof}

To obtain some positive results, it is convenient to expand our language.  For each $n\geq 1$, let $\psi_n(u)$ denote the $\la_{\os}$-formula
$$\inf_{\|x\|\leq 1}\left(\min\left\{\|\left[ \begin{matrix} u\otimes 1_n & x\end{matrix}\right]\|^2, \ \left\|\left[ \begin{matrix}u\otimes 1_n \\ x\end{matrix}\right]\right\|^2\right\}-\|x\|^2\right).$$  It was observed in \cite{GL} that if $X$ is an operator system and $u\in X$ is a contraction, then $u$ is a unitary of $X$ if and only if $\psi_n(u)^X=2$ for all $n\geq 1$.  Let $\la_{\os}^\#$ be the language obtained from $\la_{\os}$ obtained by adding new unary predicates $Q_n$ for $n\geq 1$.  Let $T_{\mathrm{C}^*,\os}^\#$ be the $\la^\#$-theory obtained from $T_{\mathrm{C}^*,\os}$ by adding, for each $n\geq 1$, the axiom
$$\sup_u|Q_n(u)-\psi_n(u)|=0.$$

\begin{prop}
$T_{\mathrm{C}^*,\os}^\#$ is $\forall\exists$-axiomatizable.
\end{prop}

\begin{proof}
If $A_1\subseteq A_2\subseteq A_3\subseteq \cdots$ is an ascending sequence of models of $T_{\mathrm{C}^*,\os}^\#$, then each $A_m$ is a \cstar-algebra and the inclusions are complete order embeddings that preserve each predicate $P_n$, whence map unitaries to unitaries.  It follows from Corollary \ref{pisier} that the inclusions are $*$-homomorphisms, whence the union is a \cstar-algebra and hence a model of $T_{\mathrm{C}^*,\os}$.  It is readily verified that the axioms of $T_{\mathrm{C}^*,\os}^\#$ not in $T_{\mathrm{C}^*,\os}$ also hold in the union, whence the union is a model of $T_{\mathrm{C}^*,\os}^\#$.
\end{proof}

In classical logic, the next corollary would be immediate.  In the continuous case, a few words are in order.

\begin{cor}
$T_{\mathrm{C}^*,\os}$ is $\forall\exists\forall$-axiomatizable.
\end{cor}

\begin{proof}
Let $\sigma:= \sup_x\inf_y \varphi(x,y)=0$ be an $\la_{\os}^\#$-axiom for $T_{\mathrm{C}^*,\os}^\#$.  We can view $\varphi$ as an $\la_{\os}$-formula by replacing any occurrence of the new predicates $Q_n$ by the $\la_{\os}$-formula $\psi_n$.  Fix $\epsilon>0$ and take a \emph{restricted} $\la_{\os}$-formula $\theta(x,y)$ that approximates $\varphi$ to within $\epsilon$.  (See \cite[Section 6]{bbhu} for the definition of restricted formula.)  The key property of restricted formulae is that they are monotone in their arguments.  Thus, an easy induction shows that any formula formed from restricted connectives and atomic and existential formulae is logically equivalent to both a $\forall\exists$- and a $\exists\forall$-formulae.  It follows that, without loss of generality, we may assume that $\theta$ is an $\exists\forall$-formula of $\la_{\os}$.  The axiom $\sup_x\inf_y(\theta(x,y)\dotminus \epsilon)=0$ is an $\forall\exists\forall$-axiom and the totality of these axioms has the same expressive power as the original axiom.
\end{proof}

\begin{rmk}
It would be interesting to identify concrete $\forall\exists\forall$-axioms for $T_{\mathrm{C}^*,\os}$; the ``quasi-sentences'' appearing in the statement of Lemma 1 come tantalizingly close to such axioms.
\end{rmk}

We can also use this extended language to give some quantifier information about the $\la_{\os}$-definable predicate $P$ defining the graph of the algebra operation.

\begin{prop}
$P$ is a quantifier-free $\la_{\os}^\#$-definable predicate in $T_{\mathrm{C}^*,\os}^\#$.
\end{prop}

\begin{proof}
If $A\subseteq B$ is an inclusion of models of $T_{C^*,\os}^\#$, then as above, the inclusions is an inclusion of \cstar-algebras, whence it preserves $P$.
\end{proof}

\begin{cor}
For each $m\geq 1$, there is a continuous function $f_m:\r^{k(m)}\to \r$ and $\la_{\os}$-formulae $\varphi_1(a,b,c),\ldots,\varphi_{k(m)}(a,b,c)$ such that each $\varphi_i$ is either an atomic $\la_{\os}$-formula or else some $\psi_n$ and such that, for every \cstar-algebra $A$ and every $a,b,c\in A_1$, we have
$$\left|d(a\cdot b,c)-f_m(\varphi_1^A(a,b,c),\ldots,\varphi_{k(m)}^A(a,b,c))\right|<\epsilon.$$
\end{cor}

We observe essentially by the same reasoning as Lemma \ref{mainlemma} that for $u, v, w\in A$ unitaries, the formula $$\varphi_{un}^A(u,v,w) := \left\|\begin{matrix} u & w \\ 1 & -v^*\end{matrix}\right\|^2 - 2$$ is absolutely continuous with respect to $d^A(u, v\cdot w)$.

\appendix

\section*{Appendix on Pisier's Linearization Trick}

The following facts follow immediately from Stinespring's Dilation Theorem; see, for example, \cite[Theorem 18]{ozawa}.
\begin{fact}\label{kadison}
Suppose that $\phi:A\to B$ is a u.c.p.\ map between \cstar-algebras.  Then for all $x,y\in A$, we have
$$\phi(x)^*\phi(x)\leq \phi(x^*x)$$ and 
$$\|\phi(y^*x)-\phi(y)^*\phi(x)\|\leq \|\phi(y^*y)-\phi(y)^*\phi(y)\|^{1/2}\|\phi(x^*x)-\phi(x)^*\phi(x)\|^{1/2}.$$
\end{fact}

\begin{cor}\label{pisier}
Suppose that $\phi:A\to B$ is  u.c.p.\ map between \cstar-algebras that maps unitaries to unitaries.  Then $\phi$ is a $^*$-homomorphism.
\end{cor}

\begin{proof}

The previous fact shows that the set $$M_\phi := \{a\in A : \phi(a^*)\phi(a) = \phi(a^*a),\ \phi(a)\phi(a^*) = \phi(aa^*)\}$$ is a C$^*$-subalgebra of $A$ on which $\phi$ is a $^\ast$-homomorphism. By assumption we have that $\mathcal U(A)\subset M_\phi$ whence $M_\phi = A$.
\end{proof}

\begin{cor}\label{forgetful}
Suppose that $\phi:A\to B$ is a linear isometry between \cstar-algebras that is also u.c.p.  Then $\phi$ is a $^*$-isomorphism.
\end{cor}

\begin{proof}

First suppose that $v\in B$ is a unitary and let $u:=\phi^{-1}(v)$.  We then have
$$1=\phi(u)^*\phi(u)\leq \phi(u^*u)\leq \phi(1)=1,$$ so $\phi(1-u^*u)=0$.  Since $\phi$ is injective, we see that $u^*u=1$.  Similarly, we have
$$1=\phi(u)\phi(u)^*=\phi(u^*)^*\phi(u^*)\leq \phi(uu^*)\leq \phi(1)=1,$$ whence we can likewise conclude that $uu^*=1$ and hence $u$ is a unitary of $A$.

 Now fix $x,y\in A$ with $\|x\|,\|y\|\leq 1$.  Since $\|\phi(x)\|,\|\phi(y)\|\leq 1$ as well, we can write $\phi(x)=\frac{1}{2}(v_1+\cdots +v_4)$ and $\phi(y)=\frac{1}{2}(v_1'+\cdots +v_4')$, each $v_i$ and $v_j'$ unitaries in $B$.  It then follows that $x=\frac{1}{2}(u_1+\cdots+u_4)$ and $y=\frac{1}{2}(u_1'+\cdots+u_4')$, with $u_i:=\phi^{-1}(v_i)$ and $u_j':=\phi^{-1}(v_j')$ unitaries of $A$.  It follows that
$$\phi(y^*x)=\frac{1}{4}\sum_{i,j}\phi((u_j')^*u_i)=\frac{1}{4}\sum_{i,j}(v_j')^*v_i=\phi(y)^*\phi(x),$$ where the second equality follows from the inequality in Fact \ref{kadison}. 
\end{proof}

\end{document}